\documentclass{amsart}
\usepackage{amsfonts}
\usepackage{amsmath}
\usepackage{amssymb}
\usepackage[utf8]{inputenc}
\usepackage{hyperref}
\usepackage{color}

\usepackage{lipsum}

\usepackage{graphicx}

\def\C{\mathbb{C}}
\def\R{\mathbb{R}}

\def\Z{\mathbb{Z}}

\def\F{\mathcal{F}}

\def\F{\mathcal {F}}

\def\V{\mathcal {V}}

\def\T{\mathbb{T}}

\def\Isom{\sf{Isom}}

\def\Aff{\sf{Aff}}
\def\Sim{\mathsf{Sim}}

\def\st{\mathsf{St}}

\def\Eins{\mathsf{Eins}}

\def\NN{\mathcal N}

\def\TT{\mathcal T}

\def\E{\mathcal E}

\def\G{\mathcal G}

\def\UU{\mathcal U}

\def\Aut{{\sf{Aut}}}

\def\O{{\sf{O}}}
\def\SO{{\sf{SO}}}

\def\SL{{\sf{SL}}}

\def\Iso{{\sf{Iso}} }

\def\det{{\sf{det}}}

\newtheorem{theorem}{{Theorem}}[section]
\newtheorem{proposition}[theorem]{{Proposition}}
\newtheorem{isom.ext}[theorem]{{Trivial isometric extension}}
\newtheorem{definition}[theorem]{{Definition}}
\newtheorem{lemma}[theorem]{{Lemma}}
\newtheorem{remark}[theorem]{{Remark}}

\definecolor{purple}{rgb}{0.65,0.12,0.94}
\definecolor{forestgreen}{rgb}{0.4,0.64,0.13}

\hypersetup{
	colorlinks=true,
	linkcolor=red,
	filecolor=magenta,      
	urlcolor=black,
}

\begin{document}
\title{On foliations admitting a transverse similarity structure}

\author [B. Flamencourt]{Brice Flamencourt}
\address{UMPA, CNRS, 
	\'Ecole Normale Sup\'erieure de Lyon, France}
\email{brice.flamencourt@ens-lyon.fr}

\author[A. Zeghib]{Abdelghani Zeghib }
\address{UMPA, CNRS, 
	\'Ecole Normale Sup\'erieure de Lyon, France}
\email{abdelghani.zeghib@ens-lyon.fr 
	\hfill\break\indent
	\url{http://www.umpa.ens-lyon.fr/~zeghib/}}

\date{\today}
\maketitle

 \begin{abstract}
 
We give a ``conceptual'' approach to Kourganoff's results about foliations with a transverse similarity structure. In particular, we give a proof, understandable by the targeted community, of the very important result classifying the holonomy of the closed, non-exact Weyl structures on compact manifolds, from which arose the notion of locally conformally product structures. We also extract from the proof several results on foliations admitting locally metric transverse connections.
 
 \end{abstract}

 \tableofcontents
 
 \section{Introduction}

Foliations on manifolds are studied using the structures carried by their transversals, where a transversal is a manifold which is at each point transverse to the foliation and which intersects each one of the leaves. However, foliations admitting a global connected transversal are quite rare, and a way to overcome this difficulty is to find a natural identification between local transversals meeting a same leaf. This is done using the so-called holonomy pseudo-group, which consists of germs of diffeomorphisms of the transversal obtained by sliding along leaves following a pre-defined path. Once we have this identification, the $G$-structures (i.e. the reductions of the frame bundle) carried by the transversal which are holonomy-invariant are of great help to understand the underlying geometry of the foliation.

Among these structures, the Riemannian ones, defining what is called a Riemannian foliation, are probably the best understood. A very detailed presentation of this particular case can be found in the book of Molino \cite{Mol}. A slightly more general case is the one of transverse similarity structures, where the Riemannian metric of the transversal is defined, only locally, up to a homothety and is preserved by the holonomy pseudo-group only up to a positive multiplicative constant, i.e. this group acts by similarities (which are often called homotheties). These structures have been studied by Kourganoff \cite{Kou}, who proved that the foliation is then Riemannian unless it is transversally flat. However, the analysis carried out by Kourganoff involved highly technical tools which have not been fully understood by the targeted public. The aim of this article is to provide a much more conceptual proof, avoiding technicalities while going further concerning the results that can be inferred.

Although we are essentially following the same broad lines, the whole idea of this new proof is to linearize the holonomy. Indeed, the holonomy pseudo-group can be understood infinitesimally as the parallel transport along leaves according to a special class of connections, called Bott connections. Yet, this is not sufficient in general to get the full picture. Being able to linearize the holonomy exactly means that the diffeomorphisms of the holonomy pseudo-group are completely determined by their one-jet at a point. In particular, we can analyze the geometry of the foliation by looking at its normal bundle endowed with the Bott connection induced by the transverse similarity structure. This linearization is done using a Haefliger structure \cite{Hea, LM}, which is more general than a foliation and corresponds to a foliation on a neighbourhood of the zero-section of the normal bundle of the initial foliation. The holonomy pseudo-group of this new foliation is then equivalent to the initial holonomy pseudo-group, but we are now in presence of a (locally)  foliated bundle, which is easier to understand. The last ingredient of this analysis is then the existence of a holonomy-invariant transverse connection, which allows to linearize the holonomy of the Haefliger structure.

The first motivation to study transverse similarity structures in \cite{Kou} was the so-called Belgun-Moroianu conjecture, formulated in \cite{BM}. This conjecture concerns conformal geometry, and more specifically Weyl connections on compact conformal manifolds. A Weyl connection is a generalization of the notion of Levi-Civita connection to the conformal case: it is a torsion-free connection which preserves the conformal class. The Weyl structure is said to be closed when it is locally the Levi-Civita connection of a metric in the conformal class, and exact when this property holds globally. The Belgun-Moroianu conjecture stated that a closed, non-exact Weyl connection on a compact conformal manifold is flat or has irreducible reduced holonomy. However, this was disproved by a counter-example constructed by Matveev and Nikolayevsky \cite{MN15}, who showed additionally that, in the analytic case, when the holonomy is reducible and non-flat, the universal cover of the compact manifold has a natural structure of a Riemannian product when endowed with a Riemannian metric canonically induced by the Weyl connection \cite{MN17}.

The work of Kourganoff in \cite{Kou} allowed to extend this theorem to the smooth case and has been the starting point of the study of Locally Conformally Product (LCP) structures (see for example \cite{ABM, BFM, FlaLCP,MP24}). For this reason, the comprehension of this result is really significant for the authors working on LCP structures, but, as we explained above, the technicalities of foliation theory and certain proofs left to the reader in Kourganoff's text have been an obstacle to the proper spread of this knowledge. For this reason, we try here to have a more invariant approach to the problem, using for example the tools introduced in the book of Molino \cite{Mol} in order to see transverse geometry as the examination of the normal bundle of the the foliation. Yet, when speaking of the holonomy pseudo-group, some issues arise, since, as we already discussed, the holonomy is not linear. The intervention of the Haefliger structures is then a way to avoid a choice of a particular complete transversal, which led to technicalities in the original proof.

The organization of the paper goes as follow. In Section~\ref{SectPre} we introduce the notions needed for the analysis of foliations and we state the main results of the paper about foliations and de Rham decomposition of manifolds admitting a locally metric connection. Sections~\ref{SectHea} and \ref{SectCon} are devoted to the linearization of the holonomy in our particular setting by use of a Haefliger structure and well-chosen coordinates for the transversals of the foliation. In Section~\ref{SectSim} we look at the equicontinuity domain of a foliation endowed with a similarity structure, and we adapt an example given in \cite{BMT} showing that the results we obtain here cannot be generalized when we only have an holonomy-invariant transverse connection. The equicontinuity is the key to prove that a transverse similarity structure can be transformed into a transverse Riemannian structure. Section~\ref{SectProofs} contains the proofs of the main theorems about foliations. We discuss in Section~\ref{SectWeyl} an application of the previous results, reproving the striking result of \cite[Theorem 1.5]{Kou} about the holonomy decomposition of compact manifolds with a closed, non-exact Weyl structure. Finally, we show in Section~\ref{FoliationTorii}  that one can, up to some geometrical constructions, assume that an LCP manifold fibers over a compact manifold with the typical fiber being a (flat) torus.

\section{Definitions and results} \label{SectPre}

In all this text, $M$ is a connected manifold endowed with a foliation $\F$ of codimension $q$. Our goal is to study the case where $M$ is compact and the foliation $\F$ admits a holonomy-invariant transversal Riemannian similarity structure $[g]$, that is $[g]$ is a 
class up to homothety of Riemannian metrics, defined locally on each transversal, and invariant under any holonomy transformation of $\F$. However, we do not assume that $M$ is compact or that $\F$ carries a particular structure for the moment, and we will add new assumptions throughout the text.

We recall that the holonomy pseudo-group of the foliation $\F$ is the pseudo-group of diffeomorphisms of local quotient manifolds of the foliation defined by sliding along leaves. More precisely, if $x, y \in M$ are in the same leaf of $\F$ and $c : [0,1] \to M$ is a path from $x$ to $y$ staying in the same leaf, $c$ defines an element of the holonomy pseudo-group by choosing two local transversals $T_x$ and $T_y$ at $x$ and $y$ respectively (i.e. $T_x$ and $T_y$ are manifolds which are everywhere transverse to the foliations and $x \in T_x$ and $y \in T_y$). If $T_x $ and $T_y$ are small enough, one can divide $c$ into sub-paths which are each contained in the domain of a foliated chart of $(M, \F)$. In each of these domains, the sub-path of $c$ is canonically lifted to all the leaves of the domain and sliding along the leaves just means following these lifts. Sliding the points of $T_x$ until we reach $T_y$ defines a local diffeomorphism from $T_x$ to $T_y$. This does not depend on the chosen sequence of foliated charts we used. A more detailed discussion about the holonomy pseudo-group can be found in \cite{Mol}.
 
 \subsection{Foliations with transverse holonomy-invariant connection}

A transverse connection $\nabla$ on $M$ is a linear connection on the normal bundle $N \F := TM / T\F$ of the foliation $\F$, such that for any vector fields $(X, \bar Y) \in T \F \times N \F$, $\nabla_X \bar Y$ coincides with the projection of $[X,Y]$ on $N \F$, where $Y \in TM$ is any representative of $\bar Y$. It means that the tangential part of the connection is already given. Such connections are also called {\em Bott connections}. This induces a connection on any transversal of $\F$ in a natural way.

This connection is {\em holonomy-invariant} if for any $x,y \in M$, any transversals $S_x$ and $S_y$ at $x$ and $y$ respectively and for any holonomy map $\gamma$ sending $S_x$ to $S_y$ (up to a restriction of the definition sets), one has $\nabla \vert_{S_x} = \gamma^* (\nabla \vert_{S_y})$. This amounts to saying that $\nabla$ is projectable, i.e. it projects to a connection on the local quotient manifolds of the foliation (see \cite[Lemma 2.3]{Mol}).

\begin{remark}
Note that the torsion of a transverse connection is well-defined because one can construct a {\em transverse fundamental form} $\theta_T$ from the transverse frame bundle $\mathrm{Fr}(N \F)$ to $\R^q$, also called a {\em solder form}, defined at the point $p \in \mathrm{Fr}(N \F)$ by:
\begin{align}
(\theta_T(X))_p := p^{-1} (\pi_*(X)_p) && \forall X \in T\mathrm{Fr}(N \F),
\end{align}
where $p$ is seen as a map from $\R^q$ to $N \F$ and $\pi$ is the canonical projection $\mathrm{Fr}(N \F) \to M$. The torsion is the covariant exterior derivative of $\theta_T$.
\end{remark}

\begin{definition}
We consider any Bott connection. This induces a connection of the principal bundle $\mathrm{Fr}(N \F)$. The distribution on $\mathrm{Fr}(N \F)$ containing the horizontal tangent vectors which project into $T \F$ is involutive and it induces a foliation called the {\em lifted foliation}. This definition does not depend on the chosen connection (because the connection is already given along $T \F$).
\end{definition}

The lifted foliation is right-invariant by the principal group action of the frame bundle and has the same dimension as the original foliation. It can also be defined as the distribution containing all the vectors $X \in T \mathrm{Fr}(N \F)$ such that $\iota_X \theta_T = 0 = \iota_X d \theta_T$, where $\theta_T$ is the fundamental transverse form. A detailed discussion can be found in \cite{Mol}.

\begin{definition}
A transverse $G$-structure on $(M, \F)$ is a $G$-reduction, say $P_G$, of $\mathrm{Fr}(N \F)$ which is invariant by the lifted foliation (i.e. the lifted foliation is tangent to $P_G$).
\end{definition}

We call a {\em transverse metric} on $(M, \F)$ any Riemannian bundle metric on the normal bundle $N \F \to M$ of $\F$. A transverse Riemannian metric $g_T$ on $(M, \F)$ is said to be holonomy-invariant if, when one denotes by $g$ the degenerate metric on $TM$ given by the composition of the projection $TM \times TM \to N \F \times N \F$ together with $g_T$, then $\mathcal L_X g = 0$ for any $X \in \Gamma(T \F)$. Endowed with such a structure, it exists a unique torsion-free transverse connection which preserves $g_T$, called the {\em transverse Levi-Civita connection of $g_T$}. This connection is projectable, i.e. holonomy-invariant. Equivalently, a transverse holonomy-invariant metric is a transverse $O(q)$-structure.

Since Riemannian holonomy-invariant metrics are significant structures on foliations, they deserve a name:

\begin{definition}
A Foliation admitting a transverse holonomy-invariant Riemannian structure is called a {\em Riemannian foliation}.
\end{definition}

\begin{remark}
The informed reader who knows the classical book of Molino \cite{Mol} should be aware of a small difference we make in the vocabulary used here. Indeed, what we call a transverse holonomy-invariant metric is what Molino simply names a transverse metric. The reason we make this difference is that we may subsequently talk about non-holonomy-invariant objects.
\end{remark}

We introduce a particular class of connections on manifolds:

\begin{definition} \label{locmetric}
A linear connection $\nabla$ on a manifold is said to be {\em locally metric} if its torsion vanishes and its reduced holonomy group $\mathrm{Hol}_0(\nabla)$ is compact.
\end{definition}

We have an equivalent notion for transverse connections:

\begin{definition} \label{transverselocmetric}
A transverse connection is said to be {\em locally metric} if its torsion vanishes and its reduced holonomy group $\mathrm{Hol}_0(\nabla)$ is compact.
\end{definition}

\begin{remark} \label{transverse irrlocallymetric}
A transverse locally metric connection induces a Riemannian bundle metric on the pull-back of the normal bundle $N \F \to M$ to the universal cover $\tilde M$ of $M$ (which is actually the normal bundle $T \tilde M/ T \tilde \F \to \tilde M$ of the pulled-back foliation $\tilde \F$). Indeed, Since $\mathrm{Hol}_0(\nabla)$ is compact, it is conjugated to a compact subgroup of the orthonormal group $O(q)$ where $q$ is the codimension of $\F$. Consequently, denoting by $\tilde \nabla$ the lift of the connection, the $\tilde \nabla$-holonomy bundle of an arbitrary point in the frame bundle $\mathrm{Fr}(T \tilde M/ T \tilde \F)$ is a reduction of the structure group included in $O(q)$ and it is invariant by the $\tilde \nabla$-holonomy. This implies $\tilde \nabla$ preserves an $O(q)$-subbundle of $\textrm{Fr}(T \tilde M/T \tilde \F)$, and it is also a transverse connection. In addition, this subbundle is invariant by the lifted connection. Altogether, this means that $\nabla$ is the Levi-Civita connection of a transverse Riemannian metric $g_T$. If we assume moreover that this connection is irreducible, since any element $\gamma$ of $\pi_1(M)$ act on $(\tilde M, \tilde \F, \tilde \nabla)$ by transverse affine transformations, then $\gamma^* g_T = \lambda g_T$ for some $\lambda > 0$, and $\pi_1(M)$ acts on $(\tilde M, \tilde \F, g_T)$ by transverse similarities.

The transverse metric then defined on $T \tilde M/ T \tilde \F$ is always holonomy-invariant by definition of a transverse connection.

Of course, the same applies in the easier situation where we consider a (non-transverse) locally metric connection $\nabla$, i.e. there exists a metric $h$ on the universal cover $\tilde M$ of $M$ such that the lifted connection $\tilde \nabla$ is the Levi-Civita connection of $h$.
\end{remark}

We recall the definition of a de Rham decomposition:

\begin{definition}
A de Rham decomposition of a Riemannian manifold $(M,g)$ is a product of manifolds $(M_0, g_0) \times \ldots \times (M_p, g_p)$ isometric to $M$ such that $(M_0, g_0)$ is flat and the other factors are irreducible.
\end{definition}

A direct consequence of the study of transverse similarity structures we will carry out is the following general result:
 
 \begin{theorem}  \label{No Flat} Let $\F$ be a foliation with a transverse holonomy-invariant locally metric connection on a compact manifold. If the transversal connection has no
 flat factor in its local de Rham decomposition, then $\F$ is a Riemannian foliation.
 
 \end{theorem}
 
In the general case, consider a transversal $\tau$, and its de Rham decomposition $\tau = \tau^0 \times \tau^*$,  where $\tau^0$
 corresponds to the flat factor, and $\tau^*$ corresponds to all the other factors (we will precise what me mean exactly by this decomposition in Section~\ref{deRhamsection}). By taking the inverse images by the local submersions defining $\F$ of 
 $\tau^0$ and $\tau^*$, one gets foliations $\F^0$ and $\F^*$, with  transversal holonomy-invariant metric connections modeled on 
 $\tau^*$ and $\tau^0$ respectively.  For instance, $\F^0$ is obtained by saturating $\F$ with $\tau^0$, and similarly for 
 $\F^*$, thus $\F$ appears as the intersection of $\F^0$ and $\F^*$. 
 
 Observe that the previous theorem applies to $\F^0$, implying:
 
 \begin{theorem}  \label{With Flat} Let $\F$ be a foliation on a compact manifold together with a transverse-holonomy-invariant metric connection, and let $\F^0$ be the saturation of 
 $\F$ by the flat factor of the de Rham decomposition of the transversal connection. Then $\F^0$ is a Riemannian foliation.
 
 \end{theorem}

 \subsection{Foliations with a transverse holonomy-invariant similarity structure}

In the course of our analysis, we will have to deal with closures of leaves of the foliation $\F$. However, such a closure is not necessarily a manifold, since there could be a subset where the leaves accumulate. In order to overcome this technical difficulty, we consider the more general concept of {\em lamination}. A lamination on a topological space $X$ is a collection of charts $(U_i, \varphi_i)$, where $U_i$ is a covering of $X$, such that the maps $\varphi_i$ are homeomorphisms from $U_i$ to $V_i \times X_i$, with $V_i$ a subset of an Euclidean space and $X_i$ a subspace of $X$, and the transition maps $\varphi_i \circ \varphi_j^{-1}$ preserve the Euclidean factor.

If the topological space $X$ admits a distance $d$, we define the bi-equicontinuity (or we could say bi-lipschitz) domain of the lamination as the set of points $x \in X$ such that there exists $\delta(x) > 0$, for any holonomy map $\gamma$ from a transversal $T_x$ of $x$ to a transversal $T_y$ at a point $y \in X$, for any $z_1, z_2 \in T_x$,
\begin{equation} \label{lamequi}
\frac{1}{\delta (x)} d(z_1,z_2) \le d(\gamma(z_1), \gamma(z_2)) \le \delta(x) d(z_1,z_2).
\end{equation}
This set is more often called the equicontinuity domain, but we use this terminology in order to avoid confusion in the sequel, adopting the convention that {\em bi-equicontinuity} means that we have both the lower and the upper bound, while {\em equicontinuity} means we only have the upper bound. 
We now define precisely the structure we will be interested in, namely a transverse holonomy-invariant (Riemannian) similarity structure.

\begin{definition} \label{defSimStruct}
A transverse similarity structure on the foliated manifold $(M, \F)$ is a maximal open covering $\{ U_i \}_{i \in I}$ of $M$ together with, for any $i \in I$, a set of homothetic metrics $G_i := \{\lambda g_i, \ \lambda > 0\}$ where $g_i$ is a transverse Riemannian metric on $(U_i, \F\vert_{U_i})$ with the property that for any $i,j \in I$, $G_i\vert_{U_i \cap U_j} = G_j\vert_{U_i \cap U_j}$. This transverse similarity structure is holonomy-invariant if the sets $G_i$ are preserved by the holonomy pseudo-group, i.e. for any holonomy map $\gamma$ defined from a connected transversal $T^i \subset U_i$ to a connected transversal $T^j \subset U_j$, $\gamma^* (g_j \vert_{T^j}) = \lambda g_i \vert_{T^i}$ for some $\lambda \in \R$.
\end{definition}

A holonomy-invariant transverse similarity structure induces a natural holonomy-invariant transverse connection. Indeed, on a sufficiently small open foliated domain on which one can choose a globally defined holonomy-invariant transverse metric of the similarity class, and the transverse Levi-Civita connection of this metric is actually independent of this choice. Consequently, if one defines a connection in such a way around each point, we just remark that these connections coincide on the intersections of the open sets because of the compatibility condition of Definition~\ref{defSimStruct}.

The main goal of our study of foliations is the following structure theorem:
 
 \begin{theorem} \label{Similarity} Let $\F$ be a foliation on a compact manifold, endowed with a holonomy-invariant similarity structure. If $\F$  contains a closed invariant subset where it is an equicontinuous lamination, then $\F$ is a Riemannian foliation. 
 
 If $\F$ is not Riemannian, then it is transversely flat, i.e. transversely modelled on $ (\Sim(\R^q),  \R^q)$.

 \end{theorem}

\begin{remark}
Observe that we are making use above of a ``metric'' equicontinuity notion rather than a ``topological''  one,  to mean that we have here Lipschitz estimates instead of just a rough uniform modulus of continuity.  This  was to simplify exposition, and  also because of equivalance of these equicontinuity concepts in our framework of  foliations endowed with a transverse holonomy-invariant  connection,    and   again as we will see it later in Section \ref{Equi.Riem.},   this is equivalent to be Riemannian (in the classical sense). In the general case,  variants of ``topological Riemannian'' foliations were relatively recently introduced   and studied  from the point of view of their leaf closures and their  relationships with the classical ``smooth Riemannian'' foliations. The history started  with questions asked by E. Ghys \cite[Appendix E]{Mol}, and some answers by Kellum \cite{Kel}. As more recent references, we can quote: \cite{AC,ALC,ALMG}. 
\end{remark}

\bigskip

\begin{remark}
Theorem  \ref{No Flat}  says that among foliations with a transverse holonomy-invariant locally metric connection, only the transversely flat case is  relevant in the sense that it  may have  a strong non-Riemannian dynamics. Here flat means the foliation is transversally modelled on $(\Aff(\R^q), \R^q)$.  This is a completely  open  research domain, even in the dimension 0  case, that is that of compact locally flat  manifolds, for which there is the huge classification conjectures by Marcus and Auslander.   A more tamed situation is that of foliations with  of  transvere flat similarity  structure, i.e. those  modelled on $ (\Sim(\R^q),  \R^q)$. The dimension 0 case is classified by Fried \cite{Frie}.  In the non-trivial codimension cases, most investigations concern the cases $q= 1$, that is $(\Aff(\R), \R)$ and $q = 2$, which actually reduces (up to orientation) to the 1-dimensional complex situation $(\Aff(\C), \C)$.  Let us quote the following works around this problem: \cite{Bar, Blu79, Ghys91, IMT, Mat, Plan, Sca, Simic}
\end{remark}

 \subsection{De Rham decomposition} \label{de Rham}

The final step of our analysis is to prove the existence of a de Rham decomposition on the universal cover $\tilde M$ of a compact manifold $M$ admitting a locally metric connection $\nabla$.

The main issue is that the classical de Rham theorem does not apply since this metric is not complete in general. With these notations we have: 

 \begin{theorem} [De Rham decomposition] \label{De Rham decomposition}
 Let $M$ be a compact conformal manifold together with a locally metric connection $\nabla$. Then, the Riemannian manifold $(\tilde M, h)$ admits a de Rham decomposition, where $h$ is any metric for which $\tilde \nabla$ is the Levi-Civita connection of $\tilde \nabla$ on $\tilde M$. Furthermore, the flat  factor is complete.
 \end{theorem}

\subsection{LCP manifolds and foliation by torii} In the last two sections of this paper, we investigate we apply the previous results to compact manifolds admitting a closed non-exact Weyl structure, which essentially means that they are endowed with a similarity structure (see Section~\ref{SectWeyl} for the formal definition). We give an alternative proof of the following theorem, originally proved in \cite{Kou}:

\begin{theorem}[Kourganoff] \label{UCoverStructure}
Let $(M, c, D)$ be a compact conformal manifold endowed with a closed non-exact Weyl structure. Let $\tilde M$ be the universal cover of $M$ and let $h$ be the metric (unique up to a positive multiplicative constant) whose Levi-Civita connection is the lift of the Weyl connection. Then one of the following cases occurs:
\begin{itemize}
\item $(\tilde M, h)$ is flat;
\item $(\tilde M, h)$ is irreducible;
\item $(\tilde M, h)$ is isometric to $\R^q \times (N, g_N)$ where $q \ge 1$ and $(N, g_N)$ is an irreducible Riemannian manifold.
\end{itemize}
\end{theorem}

In the last case of Theorem~\ref{UCoverStructure}, $(M,c,D)$ is called an LCP structure. The foliation induced by the submersion $\tilde M \to N$ descends to a foliation on $M$. It was proved in \cite{Kou} that this natural foliation has leaves' closure finitely covered by torii. We go further in this direction, proving that, up some transformation, the closure of the leaves induces a fibration whose fibers are (flat) torii. More precisely:

\begin{theorem} \label{FoliationToriiThm}
Let $(M, c, D)$ be an LCP structure. Then, up to passing to  an $\SO(k)$-principal bundle over $M$, $k > 2$, inheriting an $\SO(k)$-invariant LCP structure, the closures of the leaves of the foliation defined by the flat factor are (flat) torii. More precisely, these torii are given by a (regular) fibration of $M$ over a compact manifold.
\end{theorem}

\begin{remark} The general theory of Riemannian foliations states that closures of the leaves define a ``singular'' foliation (with all leaves closed). 
Then a principal fibration  trick allows one to desingularize it and get a regular foliation whose quotient space is an orbifold, that is the closure foliation is a ``Seifert fibration''. Our result here says that  the quotient space is a genuine manifold instead of an orbifold. Also, in our case, the closure of leaves are torii, as 
in the case of 1-dimensional Riemannian  foliations by a theorem of Y. Carri\`ere \cite{Car}.

\end{remark}
  
\section{Construction of a Haefliger structure} \label{SectHea}

From now on, we assume that the manifold $M$ is compact. \\

 The usual definition of the  holonomy starts with taking a  complete transversal submanifold $\TT$ to $\F$ and considering the holonomy transformations as local diffeomorphisms of $\TT$. One can take for instance $\TT$ to be a union of small local transversals in a family of flow boxes covering $M$. Changing $\TT$ to $\TT^\prime$ induces an equivalence between the two pseudo-groups of local diffeomorphisms of $\TT$ and $\TT^\prime$. All this, is quite delicate to formalize and manipulate. There is in particular the problem of artificial blow up of the holonomy  maps due to the choice of $\TT$ (in particular when approaching the $\TT$-boundary).  There is however the nice situation where $\F$ admits a supplementary foliation, and holonomy maps are viewed as local diffeomorphisms of this foliation. A strong   simple and ``talking''  sub-case is that  of foliated (also said flat) bundles where the holonomy is globally defined.   Here, $M$ fibers over a typical leaf $F$ with fiber $\TT$ and the foliation $\F$ is transversal to fibers. If $c$ is a path in $F$ with endpoints $x, y$, then the  associated holonomy is a map  $H_c: \TT_x  \to \TT_y$ given by lifting $c$ tangentially to $T \F$. In other words, $T \F$  as a supplementary of the vertical space of the fibration is seen as the horizontal space of a connection, and since it is integrable, this  (non-linear) connection is flat, and $H_c$ is the holonomy of its connection. The holonomy is given by a representation $\rho: \pi_1(F) \to \textrm{Diff}(\TT)$ and the total space is the suspension of $\rho$, namely the quotient of $\tilde{F} \times \TT$ by the diagonal action of $\pi_1(F)$.
 
In fact, Haefliger structures (even though motivated by other considerations), exactly allows one to maneuver to bring himself  back to this situation of foliated bundles, but locally. For a detailed exposition about these structures, see for example \cite[Section 1.3]{LM} and the references therein.

We will give the technical details in the subsequent lines, but the overall idea is as follows. We would like to associate to the foliation $\F$ a new foliation which is locally a flat bundle as described above, and whose holonomy group is equivalent to the one of $\F$ with a well-chosen identification. In order to do so, we define the new foliation not on the manifold $M$ but on a neighbourhood of the null-section of the normal bundle $\NN$ of $\F$ in $TM$. The word ``locally'' does not mean that we can locally trivialize the foliation here (since this is always the case), but rather that the foliated manifold we consider fibers over a typical leaf and this foliation is locally trivial only if we takes the fiber small enough around a point (see Figure~\ref{local}). We want the fibers of the fibration $\NN \to M$ to be transversals of the new foliation, so that the holonomy group would obviously be the same as the one of $\F$. We thus need a way to identify the normal bundle with transversals of the foliation $\F$: this can be done by taking the exponential of any Riemannian metric on $M$ restricted to the normal bundle of $\F$, and the new foliation will be the pull-back of $\F$ by this map. This construction has the effect of changing the dimension of the foliation without altering the co-dimension, but it also gives natural transversals, fixed once and for all, at each point $x$ of the null-section of $\NN$, since the normal bundle at $x$ is such a transversal. This transversal is moreover an open neighbourhood of zero in a vector space and this will allow us to define naturally foliations admitting a transverse invariant connection as those for which the holonomy is linear (see the beginning of Section~\ref{SectCon}).

\begin{figure}
\includegraphics[scale = 0.4]{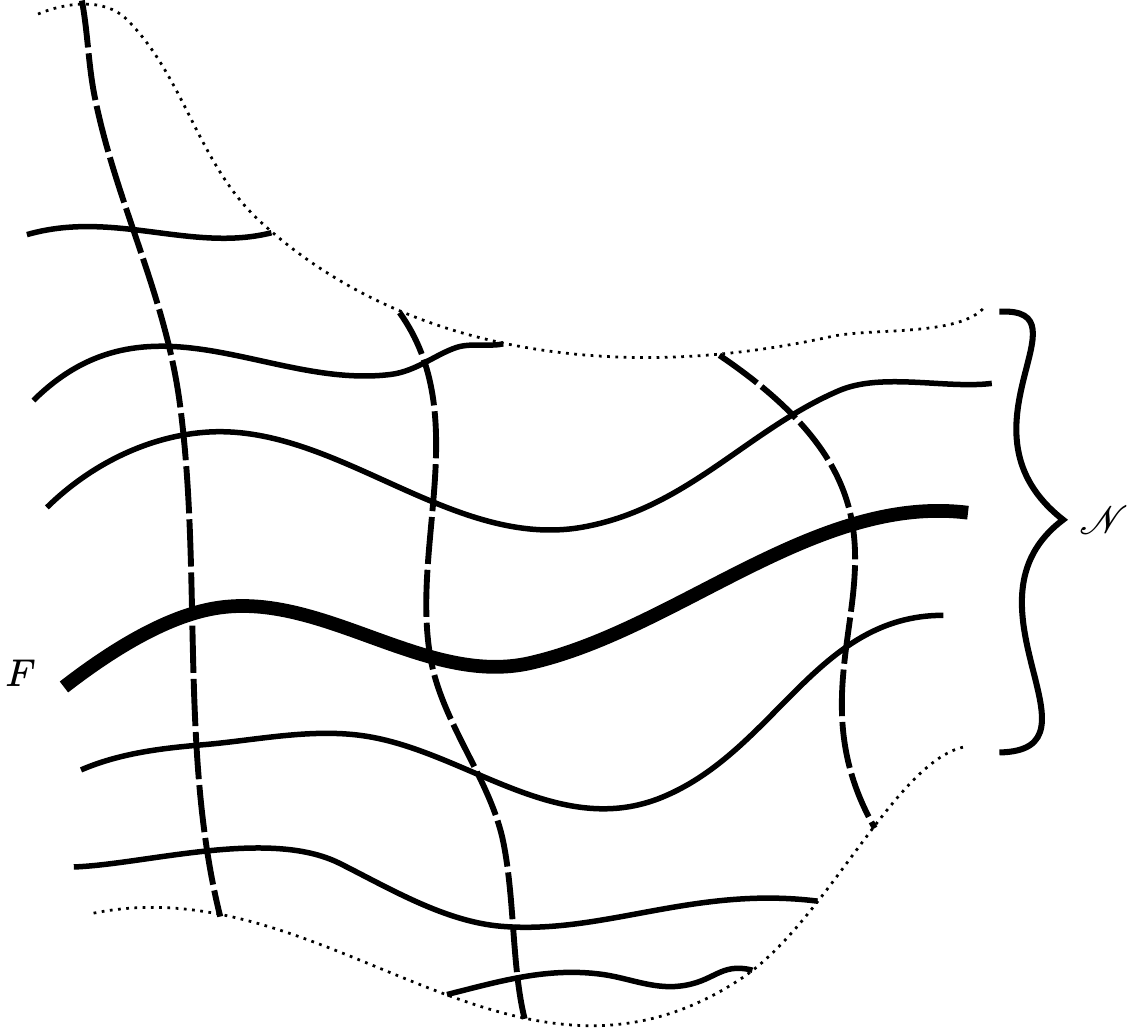}
\caption{An example where the foliation is only locally flat. The manifold $\NN$ fibers over the typical leaf $F$. The leaves are the horizontal curves and the vertical curves are the fibers of the fibration, however this fibration can be locally trivialized only on a neighbourhood of $F$, since some leaves are not ``complete".}
\label{local}
\end{figure}

Let us now give the actual construction. We endow $M$ with an auxiliary Riemannian metric $k$, which will be often implicit and whose choice does not matter.

We consider $\NN := T \F^\perp \simeq T M / T \F$, the normal bundle of $\F$. By compactness of $M$, there exists an open neighbourhood $O$ of the zero-section of $\NN$, such that the exponential Riemannian map of $k$, denoted by $\phi: O \to M$, once restricted to the fiber over a point $x \in M$, i.e. $\phi_x: O_x = O \cap \NN_x \to M$, is a diffeomorphism onto a transversal $\TT_x$ to $\F$ passing through $x$.

We can then consider $\hat{\F} := f^* \F$, the pull back by $\phi$ of $\F$.  This $\hat{\F}$ is  transversal to the fibers of the fibration $\NN \to M$ (but only locally, i.e. along $O$), and thus this looks like  to a foliated (flat) bundle.

The $\hat{\F}$-leaves are tubular neighbourhoods of the $\F$-leaves.  More precisely, $\F$ is gotten as the intersection of $\hat{\F}$ with the 0-section (of the fibration $\NN \to M$), and each $\hat{\F}$-leaf retracts naturally to a $\F$-leaf which is its intersection with the $0$-section. For this, when speaking of holonomy of $\hat{\F}$, we can restrict ourselves to $\F$-paths, i.e $\hat{\F}$-paths contained in $M$. We will in fact often identify $x \in M$ with its image $0_x$  by the $0$-section.

The philosophy is that $\F$ and $\hat{\F}$  have the same holonomy maps, which can be seen as local diffeomorphisms between the  family of $\hat{\F}$-transversals $\{O_x\}_{x \in M}$, or alternatively between the family of $\F$-transversals $\{\TT_x\}_{x \in M}$. More precisely, if $x,y \in M$ are in the same leaf $F$ of $\F$ and $c$ is a path in $F$ joining $x$ and $y$, the holonomy map induced by $c$, denoted by $\bar H_c : \TT_x \to \TT_y$ coincides with $\phi_y \circ H_c \circ \phi_x^{-1}$ defined on a neighbourhood of $0_x$ in the fiber $O_x$ to $O_y$, where $H_c$ is the holonomy map induced by $c$ on $(O, \hat{\F})$. Observe that the domain of definition of $H_c$ is not the full $O_x$, since  $H_c$ is obtained by considering horizontal lifts of $c$, but this does not necessarily exist for all time, e.g.  $O$ is not compact. \\

The {\em infinitesimal holonomy} map $h_c: \NN_x \to \NN_y$ is the derivative $d_x H_c$.  It also equals 
the usual holonomy of the Bott connection $\nabla^B$ on the normal bundle $\NN$, defined by
\begin{align}
\nabla^B_X Y := [X,Y]^\mathcal N, && \forall X \in \Gamma(TM), \forall Y \in \Gamma(\mathcal N),
\end{align}
where the $\mathcal N$-exponent stands for the projection onto the normal bundle. Indeed, this connection is flat and its holonomy bundle through $x$ is exactly the leaf containing $x$.

\section{Holonomy-invariant transversal connection} \label{SectCon}

We now assume that $(M,\F)$ carries a holonomy-invariant transverse connection $\nabla$, and we keep all the notations introduced in the previous section. This connection induces a connection $\nabla^x$ on each transversal $\TT_x$ in a natural way.

Observe that $\phi$ and accordingly $\hat{\F}$ is by no means unique, but any choice of $\hat{\F}$ shares with $\F$ the same holonomy pseudo-group (up to equivalence in this category). We will in fact keep $\TT_x$, but modify $\phi_x: O_x \to \TT_x$ to become the exponential map for $\nabla^x$. Indeed, the tangent space $T_x \TT_x$ coincides with $\NN_x$, and thus  the exponential map for  $\nabla^x$ sends a neighbourhood of $0_x$ (identified to $x$) to a neighbourhood of $x$ in $\TT_x$. So, up to a restriction of $O$, we can assume that the new $\phi$ is this exponential map.

Now, the holonomy map $\bar H_c $ sends  a neighbourhood of $x$ in $ \TT_x $ to a neighbourhood of $y$ in $ \TT_y$  and sends $\nabla^x$ to $\nabla^y$ because the transverse connection is holonomy-invariant. But a connection-preserving map becomes equal to its derivative in exponential coordinates. This means precisely that $H_c$ coincides on appropriate neighbourhoods with its derivative, the infinitesimal holonomy $h_c$. 
 
 It is also true, conversely, that $\F$ has a holonomy-invariant transversal connection if there is an associated $\hat{\F}$ having such a  ``linear'' holonomy.
 
\subsection{Equicontinuity domain} For any $x \in M$, we denote by $\Gamma_x(\mathcal F)$ the set of all paths emanating from $x$ and contained in the leaf passing through $x$.  One can define the (infinitesimal) equicontinuity (or lipschitz) domain  as the set
\[
\{ x \in M \ \vert \ \exists \alpha = \alpha_x > 0, \ \forall c \in \Gamma_x(\mathcal F), \ \parallel h_c \parallel \le \alpha \},
\]                                                                                                                                                                                                                                                                                                                                                                                                                                                                                                                                                                                                                                                                                                                                                                                                                                                                                                                                                                                                                                                                                                                                                                                                                                                                                                                                                                                                                                                                                                                                                                                                                                                                                                                                                                                                                                                                                                                                                                                                                                                                                                                                                                                                                                                                                                           
where the operator norm is defined by means of  an auxiliary   metric $k$.  However, we will need a ``bi-equicontinuity''  condition, where we also want $h_c$  to have a bounded contraction. To this purpose, we define $\E$ to be the (infinitesimal) bi-equicontinuity domain:
\[
\E = \{ x \in M \ \vert \ \exists \delta = \delta_x >1, \forall u \in O_x, \ \forall c \in \Gamma_x(\mathcal F), \ \frac{ \Vert u \Vert} {\delta(x)}  \leq \Vert h_c  (u) \Vert \leq \delta  \Vert u \Vert \}.
\]
In the similarity case, on which we will focus below,  this is equivalent to:
\[
\frac{1} {\delta(x)}  \leq \parallel h_c \parallel  \leq \delta(x)
\]
(or equivalently  $  \frac{1} {\eta(x)}  \leq   \det (h_c)  \leq \eta(x) $ where $\eta(x)$ is a power of the previous $\delta(x)$ depending on the codimension).                                                                                                                                                                                                                                                                                                                                                                                                                                                                                                                                                                                                                                                                                                                                                                                                                                                                                                                                                                                                                                                                                                                                                                                                                                                                                                                                                                                                                                                                                                                                                                                                                                                                                                                                                                                                                                                                                                                                                                                                                                                                                                                                                                                                                                                                                              

It is obvious that the bi-equicontinuity domain $\E$ is $\F$-invariant by definition. Our aim is to show that in the situation at hand, i.e. when the transversal has a holonomy-invariant Riemannian similarity structure, it is the whole manifold $M$. Since $M$ is connected, an easy strategy is to prove that $E$ is both open and closed. The openness is just a consequence of the existence of the holonomy-invariant transverse connection.
 
 \bigskip 

\begin{proposition} \label{Propagation} [Propagation of equicontinuity] The bi-equicontinuity domain $\E$ is open. Furthermore, the leaves admit compact invariant neighbourhoods in $\E$. In particular, the saturation of a compact subset in $\E$ is relatively compact in $\E$.
 \end{proposition}
 
 \begin{proof} Let $x \in \E$, $\epsilon >0$.  Consider the ball $B(0_x, \epsilon)$ (with respect to the metric $k$) and  its saturation by
 the holonomy pseudo-group of $\hat \F$:
\[
S(x, \epsilon) := \bigcup_{c \in \Gamma_x(\mathcal F)} H_c (B(0_x, \epsilon)).
\]
Since the holonomy is linear (that is $H_c = h_c$), 
and by definition of $\delta (x)$,
\[
S(x, \epsilon) \subset \bigcup_{y \in \F_x} B(0_y, \delta (x) \epsilon),
\]
hence, for $\epsilon$ small one has $\overline{S(x, \epsilon})\subset O$. The holonomy maps are thus defined for all times if one starts with  a sufficiently small ball $B(0_x, \epsilon)$, and in addition the image of the saturation $S(x, \epsilon) $ by $\phi$ is a relatively compact $\F$-invariant neighbourhood of $\F_x$. For $u \in B(0_x, \epsilon)$ near $0_x$, $u$ in the bi-equicontinuity domain of $(O, \hat \F)$ by linearity of the holonomy. Now, if as previously  announced we want to restrict ourselves to the holonomy of $\F$-leaves (instead of $\hat{\F}$), we use $\phi: O \to M$, and see that $z = \phi (u)$ belongs to the 
$\F$-bi-equicontinuity domain, which is therefore open.
\end{proof}

\section{Holonomy-invariant transversal  similarity structure} \label{SectSim}

We now endow the foliation $\F$ with a transverse (Riemannian) similarity structure in the sense of Definition~\ref{defSimStruct}. This in turn gives us a holonomy-invariant transverse connection $\nabla$ as explained in the line following the definition. In particular, all the constructions and results of the previous sections still hold, and we keep the same notations. Our goal is to prove the closeness of $\E$ under this assumption.

As above, we can assume that for any $x \in M$, $\phi_x : O_x \to \TT_x$ is the exponential map of $\nabla^x$, up to a restriction of $O$.

We can define a transverse Riemannian metric $g$ on $(M, \F)$ by taking around each point $x \in M$ a local transverse metric $g_{V_x}$ on $T \F / T M$ belonging to the transverse similarity class and restricting it to $T_x M /T_x \F$. In order to insure smoothness of this family we fix a volume element on the normal bundle $\NN \simeq N \F \to M$.

Note that $g$ is not holonomy-invariant in general, since it depends strongly on the chosen volume element.

\subsection{Singular metric}

Let $R$ be the curvature tensor of the bundle $\NN \to M$ endowed with the connection $\nabla$ (seen as a map from $\Gamma(TM) \times \Gamma(\NN)$ to $\Gamma(\NN)$). We restrict it to $\Gamma(\NN) \times \Gamma(\NN)$ and we consider the function $w : M \to \R_{\ge 0}$ which associates to $x \in M$ the norm of $R$ with respect to $g$ at $x$.

The function $w : O \to \R_{\ge 0}$ is continuous because it is the norm of a smooth function. In particular, since $M$ is compact, $w$ has a maximum $\max_M w$. The (possibly degenerate) transverse metric $m$,  defined  by $m := w g$, is a holonomy-invariant singular transverse metric. 
 
 \begin{proposition} \label{non-flat curvature}
Let $\UU = \{x \in M, \ w(x) \neq 0 \}$. Then, $\UU$ is contained in the bi-equicontinuity domain $  \E$.
 \end{proposition}

 \begin{proof}  Let $x \in \UU$ and let $c: [0, 1] \ni t \mapsto c(t) \in  \F_x$ be a path  joining $x = c(0)$ to $y := c(1)$.  Let $a :=  w(x), b := w(y)$. Assume $a > b$,  say, more precisely,  that $b = \inf \{w (c(t)), \ t \in [0, 1]\}$. The infinitesimal holonomy from $\NN_{c(t)} $ to $\NN_{c(s)}$ is a homothety of ratio $\sqrt{ \frac{w(c(t))}{w(c(s))}}$ with respect to the metric $g$. In particular, since $b$ realizes the minimum of $w$ on $c$, there exists $\epsilon > 0$ (which can be taken uniform because $M$ is compact) such that for any $t \in [0, 1]$, the $\hat \F$-holonomy $B (0_y, \epsilon) \to B (0_{c(t)}, \epsilon)$  is well defined and contracting.
 
Thus, if $w$ vanishes somewhere on $\phi(B (0_y, \epsilon))$, then the same holds on $\phi(B (0_x, \epsilon))$. Assuming $\epsilon$ is small enough, $w$ does not vanish on $\phi(B (0_x, \epsilon))$ neither on $\phi(B (0_y, \epsilon))$. In particular, the $\epsilon/2$-tubular neigbhourhood of the zero-section in $\NN$ is sent by the exponential map $\phi$ onto a subset of $M$ whose closure does not meet the null locus of $w$. Consequently, $w$ has a non-zero infimum on $\F_x$, and the homothety ratios $\sqrt{\frac{w(x)}{w(z)}}$ are bounded by two positive bounds for all $z \in \F_x$. Therefore, $x \in \E$.
 \end{proof}

  \subsection{Dichotomy}

It remains to prove that $\E$ is closed, which would lead to the following result:

 \begin{proposition} \label{Closed}
A Foliation with a transverse holonomy-invariant similarity structure is everywhere bi-equicontinuous whence 
 it is bi-equicontinuous at some point, that is,
 if $\E$ is non-empty, then  $\E = M$.
 
 \end{proposition}
 
 \begin{proof} Let $A := \bar{\E} - \E$ be the boundary of $\E$ (this a an $\F$-invariant subset of $M$), $B$ an $\epsilon$-neighbourhood of $A$ (for the metric $g$, and taking $\epsilon$ small enough) and $C := \E - B$. Thus, by Proposition 
 \ref{Propagation} the saturation $U$  of $C$ is an invariant subset whose closure $\bar U$ is a saturated compact subset of $\E$.
 
Assume by contradiction that $A$ is non-empty. Let $x \in A$. There is a maximal transverse open ball $\phi(B (0_x, \epsilon_x))$ contained in $M - \bar U$, with $\epsilon_x \leq \epsilon$.  Moreover, since $\bar U$ is compact and does not meet the boundary of $\E$, there exists $\epsilon_0 > 0$ such that $\epsilon_x \ge \epsilon_0$.  If a holonomy map is defined on $\phi(B(0_x, \epsilon_x))$, then its image is a ball $\phi(B(0_y, \epsilon^\prime))$ with $y \in A$ and  $\epsilon^\prime \leq \epsilon_y \leq \epsilon$ (since otherwise we meet points of the invariant set $U$). This implies that the holonomy map has derivative of bounded distortion $\epsilon / \epsilon_x \le \epsilon / \epsilon_0$, which does not depend on the point $x$.
 
 Now, let $c$ be a path emanating from $x$, i.e. $c(0)= x$, with $c(1)= y$. For small $t$, the holonomy from $c(0)$ to $c(t)$ is defined on the whole $\phi(B (0_x, \epsilon_x))$.  Its image is contained in $\phi(B (c(t), \epsilon_{c(t)})) \subset \phi(B(c(t), \epsilon))$.  Thus, this is defined for all $t$.
 
It follows that $x \in \E$: contradiction. This means that $\E$ is closed in $M$, and since it is also open and $M$ is connected, either $\E = M$ or $\E = \emptyset$.
 \end{proof}

\begin{figure}[!htbp]
\includegraphics[scale = 0.4]{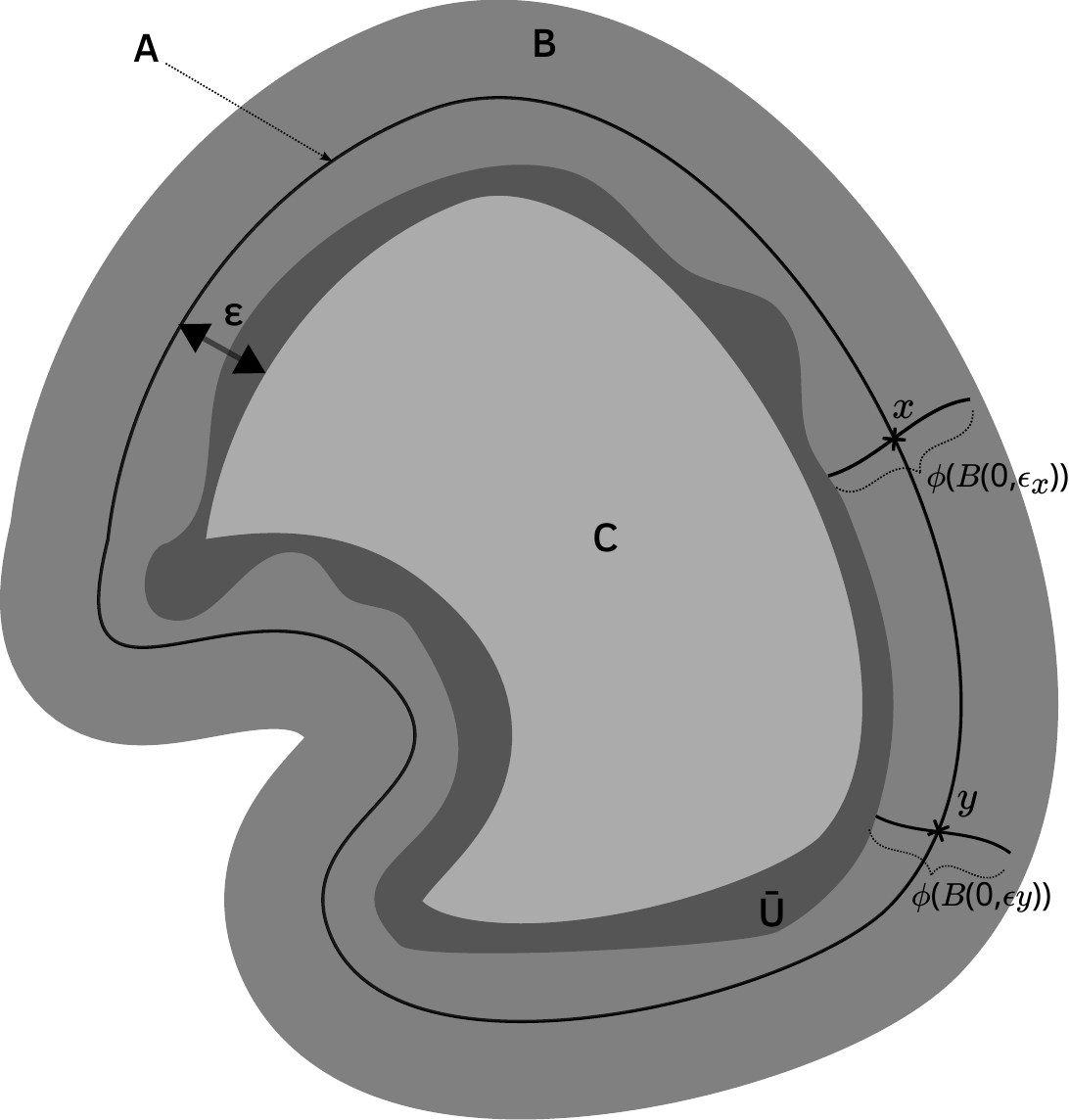}
\caption{Illustration of the proof of Proposition~\ref{Closed}.}
\end{figure}

\subsection{Counter-example in the general transversely affine case} This  dichotomy  is no longer true for general foliations with a transversal connection. An example of a transversally Lorentzian foliation of dimension 1 in a manifold of dimension 3, with a proper domain of equicontinuity, is given in   \cite{BMT}. We outline the construction here.
 
In order to give a first feeling of the example, we describe a simple construction which will not lead to a counter-example but provides the general idea. Let $\Sigma$ be a closed surface and let $f: \Sigma \to \R^{1, 1}$ be a smooth map, where $\R^{1,1}$ is the Minkowski plane endowed with the Lorentz metric  $dx dy$. Consider the map $d:  \Sigma \times \R \ni (z, t) \mapsto h^t (f(z)) \in \R^{1,1}$, with $h^t(x, y) = (e^t x, e^{-t} y)$.

Assume that $f $ is chosen so that $d$ is a submersion, hence we can consider the foliation defined by the level-sets of $d$. One has $d (z, t+ n) = \phi^n( d(z, t))$, where $\phi := h^1$, so the foliation descends to a foliation of dimension $1$ on $\Sigma \times \mathbb S^1 \simeq \Sigma \times (\R / \Z)$.

In this way, $\Sigma \times \mathbb S^1$ has a transverse structure modelled on $\R^{1, 1}$ and its cover $\Sigma \times \R$ has $d$ as a developing map, and holonomy $\phi$.  More precisely the holonomy homomorphism $\pi_1(\Sigma \times \mathbb S^1) = \pi_1(\Sigma) \times \Z \to \Isom (\R^{1, 1})$ is trivial on $\pi_1(\Sigma)$ and sends the generator of $\Z$ to $\phi$. 
 
 Recall that $d$ is defined as $d:  \Sigma \times \R \ni (z, t) \mapsto h^t (f(z)) \in \R^{1,1}$. If we want $d$ to a be a submersion, we exactly need the image of $D_zf$ to be transversal to $V(f(z))$, where $V$ is the vector field generating the one parameter group $h^t$ (So $V(x, y) = (x, -y)$).  In particular the boundary of $f(\Sigma)$ must be transverse to $V$, but this is not possible. Indeed, if we take the furthest (to $0$) non-straight curve of the $h$-flow passing through a point $f(z)$ of $f(\Sigma)$, if $D_z f$ were transverse to $V$, then there would exist a point close to $f(z)$ on a more distant flow curve. \\

We need to modify the constriction as follows. Instead of $\R^{1, 1}$, we will consider the 2-torus endowed with the product $\SL(2, \R) \times \SL(2, \R)$-action, that is the Einstein universe $\Eins^{1, 1}$ endowed with the Moebius group action. A model for this manifold is the compactification of the Minkowski space $\R^{1,1}$ together with the conformal structure given by the Lorentzian metric. More precisely, $\R^{1,1}$ embeds into the torus $T^2$ using for example the map $\varphi : \R^{1,1} \ni (x,y) \mapsto (\arctan x, \arctan y) \in \T^2 \simeq (\R/\pi \Z)^2$. The infinitesimal generator $V$ of the one-parameter group $h^t$ then induces a one-parameter group on the image of $\varphi$, still denoted by $h^t$. This one-parameter group extends uniquely to $\T^2$ since $\varphi(\R^{1,1})$ is $\T^2$ minus two circles, and is then dense. The $1$-parameter group of transformations $h^t$ has two attracting points, say  $A_1$ and $A_2$.

\begin{figure}[!htbp]
 \includegraphics[scale = 0.3]{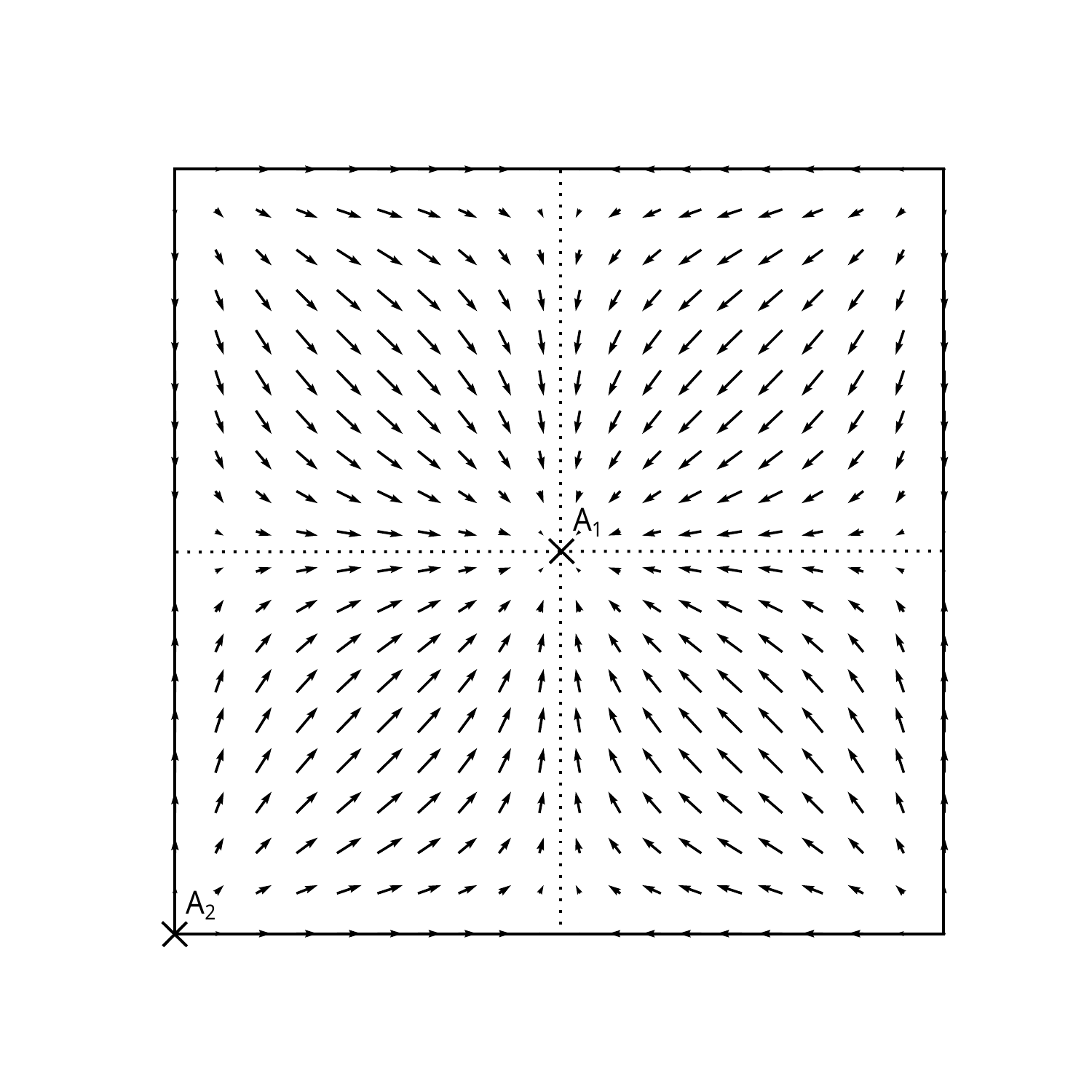}
\caption{The unfolded Einstein universe, together with the vector field induced by $h^t$.}
\label{toreplat}
\end{figure}

We claim that $h^t$ preserves a lorentzian metric on $\Eins^{1,1} - \{A_1, A_2\}$. To prove this, it is enough to find a metric  on $\R^{1,1}$, which will be conformal to the standard one and defines a metric on $\Eins^{1,1} - \{A_1, A_2\}$ by pull-back (after extension to the set where the metric is not defined). One can take $g := \frac{dx dy}{1 + x^2 y^2}$, which satisfies this condition.

Now, we can find small discs around $A_1$ and $A_2$ which are transverse to the vector field generating $h^t$. Removing them, we get a 2-punctured torus $\ddot\T^2$ with boundary $S^1 \cup S^1$.  We now consider the surface $\Sigma$ obtained by gluing smoothly two copies of this punctured torus along their boundaries. In order to define a suitable function $f$, we write $\Sigma = \ddot\T^2 \cup \ddot\T^2 \cup GZ$ where $GZ$ is the gluing zone. We define $f$ to be the canonical projection onto $\T^2$ on each copy of $\ddot\T^2$ and it is easy to see that we can chose $f$ so that its set of critical points is just two circles, projecting onto two disjoint small circles around $A_1$ and $A_2$ (see Figure~\ref{dessin} below). But such circles are transverse to the vector field generating the flow $h^t$, hence the map $d : \Sigma \times \R \ni (z,t) \mapsto h^t \cdot f(z) \in \Eins^{1,1}$ is a submersion.

\begin{figure}[!htbp]
 \includegraphics[scale = 0.25]{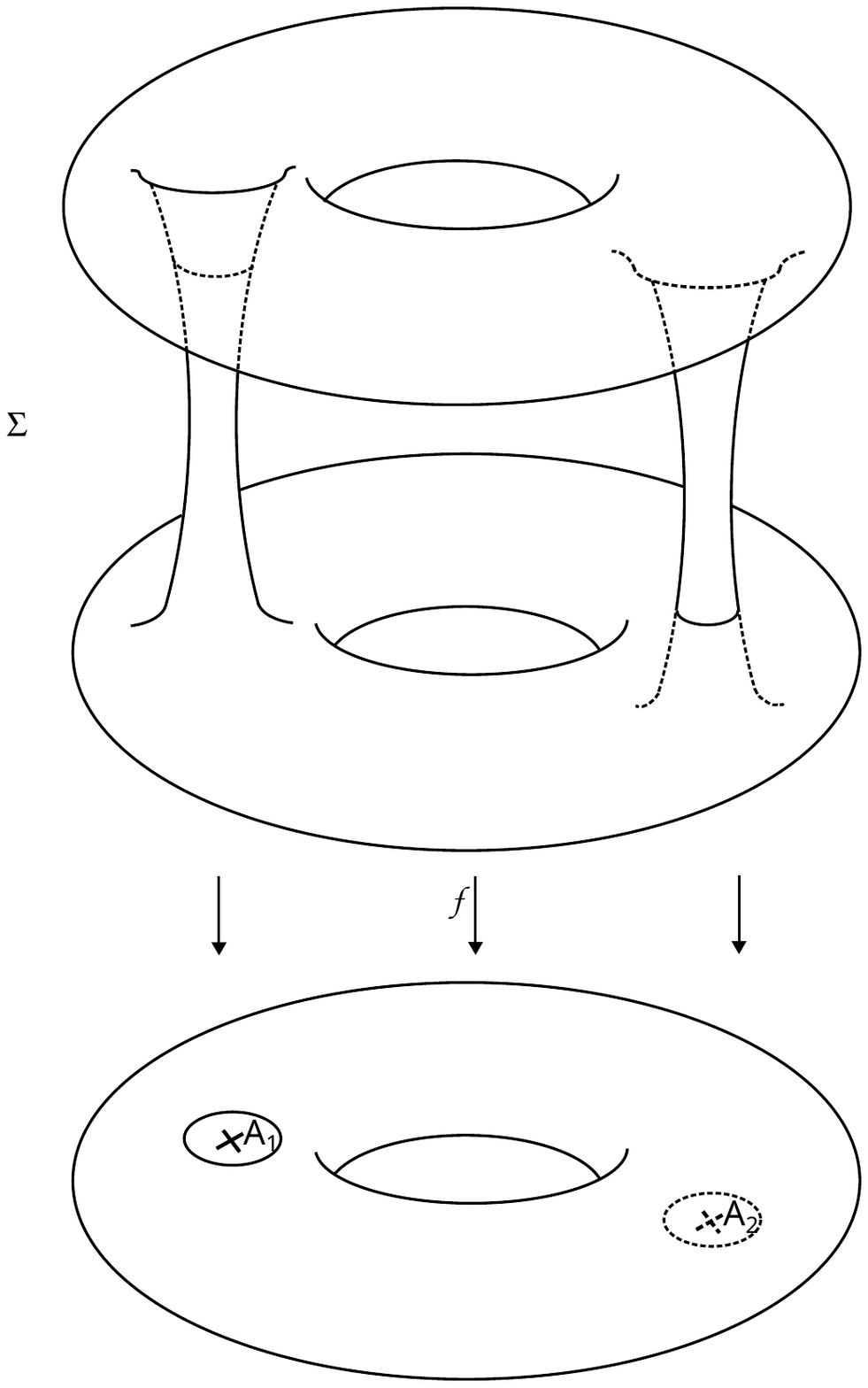}
\caption{The map $f$ from $\Sigma$ to $\Eins^{1,1}$. Note that the drawing is not faithful since there should be no self-intersection, but this is an immersion of $\Sigma$ into $\R^3$.}
\label{dessin}
\end{figure}

As before, we obtain a foliation on $\Sigma \times \mathbb S^1$ given by the level-sets of $d$ and with holonomy $\phi := h^1$. The holonomy morphism sends the generator of $\Z$ to $\phi$, as previously. The holonomy on this foliation preserves the Lorentzian metric $g$, hence it preserves the induced Levi-Civita connection on $\Eins^{1,1}$. The equicontinuity domain is a non-empty but also proper subset since it consists of all the points outside of the two points where the flow $h^t$ is hyperbolic and the straight lines of the flow emanating from them (in other words, all the straight lines drawn on Figure~\ref{toreplat}).
 
 \section{Proofs of Theorems \ref{No Flat}, \ref{With Flat}, \ref{Similarity}, \ref{De Rham decomposition}} \label{SectProofs}
 
 \subsection{Equicontinuity implies Riemannian} \label{Equi.Riem.}

We recall that the notion of equicontinuity for foliations has been defined in Equation~\ref{lamequi}.

So far, we have proved that a compact foliated manifold admitting a transverse holonomy-invariant similarity structure is either flat, or everywhere equicontinuous using Proposition~\ref{non-flat curvature} and Proposition~\ref{Closed}. In order to complete the proof of Theorem~\ref{No Flat}, it remains to show that the foliation is then Riemannian.

 \begin{theorem} \label{Equi.Riem.Theorem}
A foliation with a transverse holonomy-invariant connection on a compact manifold is Riemannian once it is equicontinuous.
\end{theorem}

Our proof will follow closely the one of \cite{Wol}. Before proceeding with the proof, we recall a few notions and a fundamental structure theorem.

\begin{definition}
On a foliated manifold $(N, \G)$, a vector field $Y$ is said to be {\em foliate} if for any $X \in T \G$ one has $[X,Y] \in T \G$.
\end{definition}

\begin{definition}
A foliation is said to be parallelizable if it admits a transverse $\{e\}$-structure, where $\{e\}$ is the trivial group with one element. The parallelism is said to be {\em transversely complete} if for any any vector $\bar X \in T M/T \F$ of the parallel basis, there exists a complete vector field $X \in TM$ which projects to $TM$.
\end{definition}

\begin{theorem}{\cite[Theorem 4.2']{Mol}} \label{TCfoliations}
Let $(N, \G)$ be a foliation of codimension $q$ admitting a transversely complete parallelism. Then, then closure of the leaves of $\G$ define a foliation $\bar\G$ of codimension $q_b$ which is induced by a submersion $\pi : N \to W$, where $W$ is a manifold. Moreover, for any $z \in W$, the foliation induced by $\G$ on $\bar \G_z := \pi^{-1} (z)$ has dense leaves and the space of foliate transverse vector fields of $(N_y, \G\vert_{\bar \G_z})$ is a Lie algebra of dimension $q - q_b$. In particular, if a transverse foliate vector field of the foliation $\G$ tangent to $\bar \G_z$ vanishes at a point, then it is zero on all $\bar \G_z$.
\end{theorem}

\begin{proof}[Proof of Theorem~\ref{Equi.Riem.Theorem}]
We consider a compact manifold $N^n$ together with a foliation of codimension $q$, endowed with a transverse holonomy-invariant connection. Let $\omega_T$ be the corresponding connection form on the frame bundle $\mathrm{Fr}(TN/T\G)$ and let $\theta_T$ be the transverse fundamental form of $\mathrm{Fr}(TN/T\G)$.

We consider the lifted foliation $\F_1$ on $\mathrm{Fr}(TN/T\G)$. It admits a transverse parallelism $\{\lambda_1, \ldots, \lambda_{n^2}, u_1, \ldots, u_n\}$ defined by
\begin{align*}
\omega_T (\lambda_i) = E_i, \ \theta_T (\lambda_i) = 0 && \textrm{and} && \omega_T (u_j) = 0, \ \theta_T (u_j) = e_j,
\end{align*}
where $(E_1, \ldots, E_{n^2})$ and $(e_1, \ldots, e_j)$ are the canonical bases of $\mathfrak{gl}_n(\R)$ and $\R^q$ respectively. This is a transverse parallelism because the connection is projectable (i.e. holonomy-invariant). In addition, this parallelism is complete because the $\lambda_i$'s are obviously complete and the $u_j$'s are represented by complete vector fields. Indeed, choose an arbitrary Riemannian metric on $N$ and identify $TN/T\G$ with the orthogonal of $T\G$. The connection $\omega_T$ defines a connection on $T \G^\perp$ which can be extended to a connection $\nabla$ on $TN$. Take the representative $\tilde u_j$ of $u_j$ given by the identification $T \G^\perp \simeq TN/T\G$, then its integral curves project on $M$ to geodesics of $\nabla$, which are defined for all times. Consequently, the foliation $\G_1$ admits a complete transverse parallelism, and we can apply Theorem~\ref{TCfoliations}: the closures of the leaves give a foliation $\bar \G_1$ induced by a submersion $\pi: \mathrm{Fr}(TN/T\G) \to W$.

In addition, the foliation $\G$ has an equicontinuous holonomy pseudo-group. This implies that for any $x \in N$, the intersection of a leaf of $\G_1$ with the fiber over $x$ is a relatively compact subset. Since $N$ is compact, the closure of a leaf of $\G_1$ is therefore compact. The action of $G := \mathrm{GL}_q(\R)$ on $\mathrm{Fr}(TN/T\G)$ sends the closure of a leaf to the closure of a leaf, so it descends to an action on $W$, and this action is proper because the closures of the leaves are compact.

We now construct a fiber bundle $E$ over $W$ whose fiber over $z \in W$ is the set of transverse foliate vector fields tangent to $\pi^{-1}(z)$ (which is a finite-dimensional vector space of dimension independent of $z$ according to Theorem~\ref{TCfoliations}). Since $G$ preserves the set of foliate vector fields tangent to the closure the leaves, $E$ admits a proper action of $G$.

Altogether, the vector bundles $TW \to W$ and $E \to W$ are both endowed with a proper $G$-action, so they admit $G$-invariant bundle Riemannian metrics. Summing these two metrics, we obtain a bundle Riemannian metric on $TW \oplus E \to W$ that we can pull-back to a bundle Riemannian metric on $T \mathrm{Fr}(TN/T\G) / T \G_1 \to \mathrm{Fr}(TN/T\G)$. The transverse metric defined this way is holonomy-invariant by definition. We reduce the fiber of $T \mathrm{Fr}(TN/T\G) / T \G_1 \to \mathrm{Fr}(TN/T\G)$ to the image of the horizontal vectors with respect to the connection $\omega_T$,  and the metric induced on this vector bundle is $G$-invariant, so it descends to a transverse holonomy-invariant metric on $(N, \G)$.
\end{proof}

\subsection{Proof of Theorem~\ref{Similarity}}

We are now in a position to complete the proof of Theorem~\ref{Similarity}. We assume that the foliation $\F$ on $M$ admits a transverse holonomy-invariant similarity structure. If there is a non-empty closed $\F$-invariant subset of $M$ where $\F$ is an equicontinuous lamination, then the equicontinuity domain of $M$ is non-empty, and it is the whole manifold $M$ by Proposition~\ref{Closed}. Applying Theorem~\ref{Equi.Riem.}, the foliation $\F$ is Riemannian.

Conversely, if $\F$ is not Riemannian, its equicontinuity domain must be empty by the contrapositions of Proposition~\ref{Closed} and Theorem~\ref{Equi.Riem.}. Using the contraposition of Proposition~\ref{non-flat curvature}, the curvature of the metric $[g_T]$ is everywhere zero, i.e. it is a flat transverse metric.

\subsection{Proof of Theorems~\ref{No Flat} and \ref{With Flat}} \label{deRhamsection}

Here, we assume that $\F$ admits a transverse holonomy-invariant locally metric connection $\nabla$. We first need to define precisely what we mean by the de Rham decomposition of the transversal. By definition, $\nabla$ is a connection on the normal bundle $TM/\F \to M$, so there is a decomposition $N \F =: T_1' \oplus \ldots \oplus T_m'$ into $\nabla$-invariant subspaces, i.e. invariant by the holonomy of the connection $\nabla$. Assume first that there is no flat factor in this decomposition. The pre-images of these subspaces in $TM$ are denoted by $T_i$ and one has:

\begin{lemma}
The distribution $T_i$ is involutive.
\end{lemma}
\begin{proof}
Let $\nabla^\F$ be any connection on $T \F$. One has $TM \simeq N \F \oplus T\F$ and we consider the linear connection $\nabla^0 := \nabla \oplus \nabla^\F$. Let $X,Y \in T_i$. Since $\nabla$ is torsion-free, the torsion of $\nabla^0$ has values in $T \F$ and we deduce that there is $Z \in T \F$ such that
\[
[X,Y] = \nabla_X Y - \nabla_Y X + Z \in T_i. \qedhere
\]
\end{proof}

Let $\F'$ be the foliation induced by the distribution $\bigoplus_{i \ge 2} T_i \oplus T \F$. The connection $\nabla$ descends to a transverse connection $\nabla'$ on the normal bundle $TM / T \F'$, which is still locally metric and is preserved by the holonomy pseudo-group of the foliation $T \F'$. Moreover, $\nabla'$ has irreducible holonomy, and using Remark~\ref{transverse irrlocallymetric}, there is a transverse holonomy-invariant Riemannian metric $g_T'$ on the universal cover $\tilde M$ of $M$ such that the lifted connection $\tilde \nabla'$ is the Levi-Civita transverse connection of $g_T$ for the lifted foliation $\tilde \F'$. Since $\pi_1(M)$ acts by transversal similarities on $(\tilde M, \tilde \F',  g_T')$, $g_T'$ induces a transverse similarity structure on $M$. For any holonomy map $\gamma$ of the foliation $\F'$, $\gamma$ lifts to a holonomy map of $\tilde \F'$ which acts as a transverse isometry, so $\gamma$ preserves the similarity structure which is then holonomy-invariant. It is also non-flat since we assume there is no flat factor in the de Rahm decomposition. It remains to apply Theorem~\ref{No Flat} to get that $\F'$ is Riemannian, and we obtain a Riemannian metric on the distribution $T_1'$.

Iterating this process for all $i$'s one after the other and summing the metric obtained this way, we get a transverse holonomy-invariant Riemannian metric for the foliation $\F$, finishing the proof of Theorem~\ref{No Flat}.

To show Theorem~\ref{With Flat}, we remark that the pre-image of the flat distribution $ $ by the projection onto $TM/ T\F$ is again involutive, and we replace $\F$ by the foliation $\F^0$ induced by this new distribution. The transverse holonomy-invariant locally metric connection descends to the new transverse structure and has no flat factor in its de Rham decomposition. We can apply Theorem~\ref{No Flat} to conclude.

\subsection{Proof of Theorem~\ref{De Rham decomposition}}

In this section, $M$ is a compact manifold admitting a locally metric connection $\nabla$. By Remark~\ref{transverse irrlocallymetric}, there exists metric on the universal cover $\tilde M$ of $M$ whose Levi-Civita connection is the lift $\tilde \nabla$ of $\nabla$. Let $h$ be such a metric. Let $T \tilde M =: \tilde T_0 \oplus \tilde T_m$ be the decomposition of $TM$ into $\tilde \nabla$-holonomy-invariant subspaces such that $\tilde T_0$ is flat and the $\tilde T_i$ for $i \ge 1$ are irreducible. If there is only one factor, there is nothing to prove, so we assume this decomposition has at least two factors. We define two transverse foliations $\tilde \F$ and $\tilde \TT$ by $T \tilde\F := \tilde T_m$ and $T \tilde\TT := \bigoplus_{i = 0}^{m-1} T_i$.

Since the elements of $\pi_1(M)$ act by affine transformations on $(\tilde M, \tilde \nabla)$, they preserves this decomposition up to a permutation of the factors. There is a finite cover of $M$ on which the permutations are trivial, and we can assume without loss of generality that $M$ is this covering. Thus the distributions $\tilde T_i$ descend to distributions $T_i$ on $M$ and the foliations $\tilde \F$ and $\tilde \TT$ descend respectively to foliations $\F$ and $\TT$ on $M$.

The connection $\nabla$ descends naturally to a transverse locally-metric connection $\nabla^T$ for the foliation $\F$. This connection $\nabla^T$ is holonomy-invariant (with respect to the holonomy pseudo-group of $\F$). Indeed, for any point $x \in M$, one can take a small neighbourhood $U$ of $x$ such that the metric $h$ descends to a metric $g$ on $U$, and by the local de Rham theorem there exists, up to a restriction of $U$, an isometry $\varphi : (U, g) \to (F, g\vert_{T F}) \times (T, g\vert_{T T})$ where the product respects locally the two foliations $\F$ and $\TT$ (i.e. $T F = T \F$ and $T T = T \TT$). Moreover, $\nabla^T$ is the (transverse) Levi-Civita connection of $g\vert_{T T}$, so it is invariant by sliding along the leaves of $\F$ in $U$, thus globally invariant by the holonomy pseudo-group.

Altogether, we have a compact manifold $M$ with a foliation $\F$ and a transverse locally metric holonomy-invariant connection $\nabla^T$. In addition, $\nabla^T$ is non-flat and irreducible, so we can apply Theorem~\ref{No Flat} to obtain a transverse holonomy-invariant Riemannian structure on $(M, \F)$, i.e. a holonomy invariant metric $g_m$ on $T_m$. Iterating this construction by taking an arbitrary $\tilde T_i$, $i \ge 1$, for $\tilde \F$ instead of $T_m$, one has Riemannian metrics $g_i$ on all $T_i$'s, which are invariant by sliding along the leaves of the other $T_j$ for $0 \le j \le m$. The singular metrics defined this way lift to metrics $\tilde g_i$ on $\tilde M$.

Assume first that the flat factor is non-trivial. We will apply the following generalization of the de Rham decomposition theorem proved in \cite[Theorem 1]{PR93}:

\begin{theorem}[Ponge-Reckziegel] \label{PR}
Let $(N,g_N)$ be a simply connected pseudo-Riemannian manifold with two orthogonal foliations $L$ and $K$. Assume that the leaves of $L$ are totally geodesic and geodesically complete and that the leaves of $K$ are totally geodesic, then $(N, g_N)$ is isomorphic to a product $(N_1, g_{N_1}) \times (N_2, g_{N_2})$ such that $L$ and $K$ correspond to the foliations induced respectively by $N_1$ and $N_2$ on $N$.
\end{theorem}

We consider the two foliations $L$ and $K$ on $\tilde M$ defined by $T L := \bigoplus_{i \ge 1} \tilde T_i$ and $T K := \tilde T_0$. We define a new metric $h'$ on $\tilde M$ by setting $h' = h\vert_{T_0} \oplus \tilde g_1 \oplus \ldots \oplus \tilde g_m$. The foliations $L$ and $K$ are orthogonal and totally geodesic because $(\tilde M, h')$ is locally a Riemannian products of integral manifolds of the two foliations. In addition, $L$ is geodesically complete because its geodesics descend to geodesics of the induced foliation on $M$ together with the Riemannian structure we constructed above, and this is geodesically complete by compactness of $M$.

Applying Theorem~\ref{PR}, the manifold $(\tilde M, h')$ is isometric to a product $(M_1, g_1) \times (M_2, g_2)$ where the foliations induced by $M_1$ and $M_2$ are respectively $L$ and $K$. Moreover, $(M_1, g_1)$ is simply-connected and complete, and the metric $g_1$ is, by construction, a local Riemannian product. Applying the usual de Rham decompostion theorem, we obtain that $M_1$ is a Riemannian product. Finally, taking the product decomposition of $\tilde M$ obtained this way together with the metric induced by $h$ on this decomposition, we obtain the de Rham decomposition we were seeking.

In the case where the flat factor is trivial, we simply apply this last step of the proof. This conclude this section.

\section{Closed non-exact Weyl structures} \label{SectWeyl}

In this section, we apply the result of Theorem~\ref{De Rham decomposition}, i.e. the existence of a de Rham decomposition, in order to prove a remarkable structure theorem for compact conformal manifolds admitting a particular connection called a {\em Weyl connection}. From now on, $M$ is a compact manifold, and we endow it with a conformal structure $c$ in the following sense:

\begin{definition}
A conformal structure on $M$ is a set $c$ of Riemannian metrics such that for any $g,g' \in c$, there exists a smooth function $f : M \to \R$ such that $g = e^{2f} g'$.
\end{definition}

Conformal manifolds admit a paricular class of connection, called {\em Weyl connections} which generalize the notion of Levi-Civita connection in the conformal setting:

\begin{definition}
A Weyl connection on the conformal manifold $(M,c)$ is a torsion-free connection $D$ such that $D$ preserve the conformal structure, i.e. for any $g \in c$ there exists a $1$-form $\theta_g$ on $M$ such that $D g = -2 \theta_g \otimes g$. The $1$-form $\theta_g$ is called the Lee form of $D$ with respect to $g$. The triplet $(M,c,D)$ is called a {\em Weyl connection}.

A Weyl structure is said to be {\em closed} if its Lee form with respect to one metric, and then to all metrics in $c$, is closed. In this case, $D$ is locally the Levi-Civita connection of a metric in $c$.

A Weyl structure is said to be {\em exact} if its Lee form with respect to one metric, and then to all metrics in $c$, is exact. In this case, $D$ is the Levi-Civita connection of a metric in $c$.
\end{definition}

We assume that there is a closed, non-exact Weyl connection $D$ on the conformal manifold $(M,c)$. Once lifted to the universal cover $\tilde M$ of $M$, the conformal structure $c$ and the Weyl connection $D$ induce a conformal structure $\tilde c$ and a Weyl connection $\tilde D$ on $(\tilde M, \tilde c)$. If $g$ is a metric in $c$, the Lee form of $\tilde D$ with respect to the lifted metric $\tilde g$ to $\tilde M$ is the pull-back $\tilde \theta_g$ of the Lee form $\theta_g$ to $\tilde M$, which is then exact. This means that $\tilde D$ is an exact Weyl connection, and there exists a metric $h$, unique up to a positive mutiplicative constant, such that $\tilde D = \nabla$ where $\nabla$ is the Levi-Civita connection of $h$. In particular, there exists $f : \tilde M \to \R$ such that $h = e^{2f} \tilde g$ and $\tilde \theta_g = df$. If we pick any $\gamma \in \pi_1(M)$, one has $\gamma^* \tilde \theta_g = \tilde \theta_g$, implying that $\gamma^* df = df$ and $\gamma^* f = f + \lambda_\gamma$ where $\lambda_\gamma \in \R$. Hence, $\pi_1(M)$ acts by similarities on $(\tilde M,h)$, and these similarities are not all isometries because otherwise $h$ would descend to $M$, but this is impossible because $D$ is non-exact.

The Weyl connection $D$ is a locally metric connection of $M$ and $h$ is a metric on $\tilde M$ whose Levi-Civita connection is $\tilde D$, so we can apply Theorem~\ref{De Rham decomposition} and we infer that $(\tilde M, h)$ has a de Rham decomposition $(M_0, h_0) \times \ldots \times (M_p, g_p)$. We assume that $(\tilde M, h)$ is neither irreducible nor flat and we denote by $(N,h_N)$ the Riemannian product $(M_1, g_1) \times \ldots \times (M_N,g_N)$.

\begin{lemma}
The Riemannian manifold $(N,h_N)$ is irreducible.
\end{lemma}
\begin{proof}
By contradiction we assume that $(N, h_N)$ is reducible, so it can be written as a product $(N_1, h_1) \times (N_2, h_2)$ where $N_1$ and $N_2$ have positive dimension. Moreover, the group $\pi_1(M)$ acts by similarities on $(\tilde M, h)$, and in particular it contains only affine maps, which then preserve the de Rham decomposition of $(\tilde M, h)$ up to a permutation of the factors. Thus $\pi_1 (M)$ preserve the decomposition $(M_0, g_0) \times (N, h_N)$ so it projects to a group $\Gamma$ of similarities of $(N, h_N)$.

Let $\gamma$ be a non-isometric similarity in $\Gamma$, which exists because $\pi_1(M)$ does not only contain isometries. We can assume, up to taking a power of $\gamma$, that $\gamma$ preserves the decomposition $(N_1, h_1) \times (N_2, h_2)$ and it can be written as $(\gamma_1, \gamma_2)$ where $\gamma_1$ and $\gamma_2$ act on $N_1$and $N_2$ respectively. If $(N_1, h_1)$ is complete, then $\gamma_1$ has a fixed point, but a manifold admitting a similarity with a fixed point is isometric to $\R^n$, which is not possible since $(N, h_N)$ does not contain a flat factor. We can thus assume that $(N_1, h_1)$ and $(N_2, h_2)$ are both incomplete.

We recall that the Cauchy border of a Riemannian manifold $\mathcal N$ is $\partial \mathcal N := \bar{\mathcal N} \setminus \mathcal N$ where $\bar{\mathcal N}$ is the metric completion of $\mathcal N$. The Cauchy border is preserved by any similarity of $\bar{\mathcal N}$, and any similarity of $\mathcal N$ extends to a similarity of $\bar{\mathcal N}$ by density.

 Since $\Gamma$ acts cocompactly on $N$, there exist compact subsets $K_1$ and $K_2$ of $N_1$ and $N_2$ respectively such that  
 $\Gamma \cdot (K_1  \times K_2) = N_1 \times N_2$.
 By compactness, there exist
 $0 < \alpha \le \beta$ such that, if we denote by $d_1$ and $d_2$ the induced distances on the metric completions of $N_1$ and $N_2$ respectively, one has $\alpha \le d_i (K_i, \partial N_i) \le \beta$ for $i = 1,2$. Then if we define, for $i = 1,2$,
\[
D_i := \{ x \in \bar N_i \ \vert \ \alpha \le d_i(\partial N_i, x) \le \beta \}
\]
one has $\Gamma \cdot (D_1 \times D_2) = N_1 \times N_2$. We choose any $(x_1, x_2) \in N_1 \times N_2$ such that $d_1 (\partial N_1, x_1) = \alpha/2$ and $d_2 (\partial N_2, x_2) = 2 \beta$, which exists by connectedness of $N$ and because $\alpha$ is a similarity of ratio different from $1$. By the cocompactness of the action of $\Gamma$, there exists $\gamma \in \Gamma$ such that $(x_1, x_2) \in \gamma (K_1 \times K_2)$ and $\gamma$ has ratio $\lambda > 0$. Since the Cauchy border of $N$ is preserved by $\gamma$, one has $\lambda \alpha \le d_1(\partial N_1, \gamma(K_1)) \le \lambda \beta$ and $\lambda \alpha \le d_2(\partial N_2, \gamma(K_2)) \le \lambda \beta$, hence $\lambda \le 1/2$ and $2 \le \lambda$: contradiction.
\end{proof}

Since we assumed that $(\tilde M, h)$ is neither irreducible nor flat, $N$ has positive dimension and $(M_0, h_0)$ is isometric to an Euclidean space $\R^q$ with $q \ge 1$. Summarizing, we proved Theorem~\ref{UCoverStructure}.

\section {Foliation by torii} \label{FoliationTorii}

In this section we study the third case occurring in Theorem~\ref{UCoverStructure}, namely we consider $(M,c,D)$ a compact conformal manifold endowed with a closed non-exact Weyl structure. Its universal cover is $\tilde{M}$ and Theorem~\ref{UCoverStructure} gives us a Riemannian metric $h$ on $\tilde M$ such that the lift of $D$ to $\tilde M$ is the Levi-Civita connection of $h$ and $(\tilde M, h)$ is isometric to $\R^q \times (N, g_0)$ where $(N, g_0)$ is irreducible and incomplete.

Let $P$ be the projection of $\pi_1(M)$ on $\Sim(N)$,  $\bar{P}$ its closure and $\bar{P}^0$ its identity component. It was proved in \cite{Kou} that ${\bar P}^0$ is abelian. Then, the foliation induced by the submersion $\tilde M \simeq \R^q \times N \to N$ descends to a foliation $\F$ on $M$ and one can prove that the closures of the leaves of $\F$ are finitely covered by flat tori.

Our goal in this section is to prove that we actually have, up to some modifications on the manifold $M$, a foliation by torii, i.e. the foliation by the closures of the leaves of $\F$ can also be taken to be non-singular.

\subsection{How to make $\Sim(N)$ and hence $\bar{P}$ acting freely on $N$}

Let $\st(N)$ be the space of conformal frames of $N$. This is a principal $(\R^* \times \SO(d))$-bundle $(d = \dim N$) (Here $\st$ stands for Steifel).
 We see it as a subset of the  space of $d$-systems of vectors 
$\{v_1, \ldots, v_d\}$,  that is the sum $(TN)^d := TN \oplus \ldots  \oplus TN$ of $d$-copies of $TN$.  This bundle 
has a natural metric for which $\Sim(N)$ acts by similarity.  Indeed, its tangent space has a  horizontal-vertical splitting 
$\mathcal H \oplus \V$, where $\mathcal H$ is the horizontal given by the Levi-Civita connection of $g_0$, and 
$\V$ is the vertical space of the fibration. One endows $\mathcal H$ with the metric making the projection a Riemannian submersion, i.e
it sends isometrically  $\mathcal H$ to $TN$. The vertical  at a point $x \in N$ is identified to $T_xN \oplus \ldots \oplus T_xN$, and is endowed 
with the product metric.

One sees that $\Sim(N)$ acts by similarity on $(TN)^d$ and hence also on $\st(N)$.

Now, $\Sim(N)$ acts properly on $N$ and preserves a metric $g_1 $ conformal to $g_0$. Thus, $\hat{N}$, the space of $g_1$-orthonormal  frames is a $\O(d)$-bundle over $N$ and a sub-bundle of $(TN)^d$. It is preserved by  $\Sim(N)$ and $\Sim(N)$ acts on it by similarities.

We remark that the case where $\dim N = 2$ was classified in \cite{Kou}, and in this case $\bar P$ already acts freely on $N$, so the construction we are currently investigating is not needed. In the case $\dim N > 2$, the fundamental group of $\hat N$ is either $\{ 0 \}$ or $\mathbb Z / 2 \mathbb Z$, thus $\hat N$ is finitely covered by its universal cover and we can assume that $\hat N$ is simply connected without modifying too much the group acting cocompactly on $\tilde M$, i.e. $\pi_1(M)$.

The idea is to replace $\tilde{M} $ by $\hat{ \tilde {M}} =  \R^q \times \hat{N}$, and $M$ by 
$\hat{M} $, the quotient of $\hat{ \tilde {M}}$ by the $\pi_1(M)$-action. So, $\hat{M}$ is a principal $\SO(d)$-fibration 
over $M$.

\subsection{How to get rid of compact normal  subgroups of $\bar{P}$} In this subsection we show that one can restrict to the case where ${\bar P}^0$ is $\R^k$ for some $k$. This construction will not be used in the sequel, but it is interesting since it reduces the problem of the classification to a better-understood setting \cite{FlaLCP2}.

Assume  as above that $\bar{P}$ acts freely on $N$ and that $\bar{P}$ has a
compact  normal subgroup $K$. Then, we propose to modify $N$ to 
$\underline{N}$ which is the quotient of $N$ by $K$. This acts freely and properly \cite[Lemma 4.9]{Kou}, and thus the quotient is a manifold, and $\bar{P}/K$ acts naturally on this quotient. Moreover, this action is free since $K$ acts trivially on $\underline{N}$.

\subsection{Foliation by torii} 
Assume first that $\bar{P}$ acts freely (if not replace $N$ by $\hat{N}$).
Let $\F$ be the foliation of $M$ given by the factor $\R^q$. The closure of leaves of $\F$ is a singular foliation $\bar{\F}$ of $M$. 
Its lift to $\R^q \times N$ is given by orbits of the group $Q= \R^q \times \bar{P}^0  \subset \Sim ( \R^q \times N)$. 

It is proved in \cite{Kou} that $\bar{P}^0$ acts isometrically ($\bar{P}^0 \subset \Iso(N)$) and that $\bar{P}^0$ is abelian. The orbits
of $Q$, when equipped with their induced metric,  are thus homogeneous under the action of an  abelian 
group and are thus flat.

Let $f = (A, B) \in \pi_1(M)$ and assume it preserves an orbit $\R^q x  \times \bar{P}^0 y$. Up to composition with the inverse of an element $(a, b) \in \R^q \times \bar{P}^0$, we get $By = a y$, and thus $B \in \bar{P}^0$ since $\Sim(N)$ acts freely. From the fact that  $\bar{P}^0$ acts isometrically, we infer  that $f \circ a^{-1}\in \Iso (\R^q) = \O(q) \ltimes \R^q$.  So $A(x) = R(x) + u$, for some $R \in \O(q), u \in \R^q$. From this, one deduces that $f$ preserves any orbit $\R^q x^\prime \times \bar{P}^0 y^\prime$. Now, by Bieberbach Theorem, $\Gamma_0 = \pi_1(M) \cap (\R^q \times \bar{P}^0)$ is a lattice in $\R^q \times \bar{P}^0$, and all rotational parts $R$ belong to a finite group $H$ (the holonomy group). 
 
 Observe that  $\bar{P}^0$ is normal in $\bar{P}$, and hence $P \cap \bar{P}^0$ is normal in $P$, and so $\pi_1(M) \cap (\R^q \times \bar{P}^0)$, which equals $\Gamma_0$ is normal in $\pi_1(M)$. We thus have a representation by conjugacy $\rho: \pi_1(M) \to \Aut(\Gamma_0)$.  Since $\pi_1(M)$ is finitely generated, by Selberg Lemma, we can replace $\pi_1(M)$ by a finite index subgroup such that $\rho (\pi_1(M))$ is torsion free.  Observe that for $f = (A, B)$ as above, $\rho(f)$ coincides with the $R$-action on $\Gamma_0$. It has finite order and acts trivially on $\Gamma_0$ iff $R $ is trivial.  We deduce that in $M$,  all orbits are torii.

Finally, we observe that, in the general case, too, where $\bar{P}$ might (a priori) act non-freely, and after this finite index modification of $\pi_1(M)$,  we have a foliation by torii. Indeed, by the first step applied to $\hat{M}$, we know that $\pi_1(M)$ meets $  (\R^q \times \bar{P}^0)$ along a lattice $\Gamma_0$, and we will still assume $\pi_1 (M)$ acts by conjugacy on $\Gamma_0$, whithout torsion.  If $f \in \pi_1(M)$ preserves some orbit $\R^q  x\times \bar{P}^0 y$, then it induces an isometry acting trivially on $\Gamma_0$. Since $\Gamma_0$ projects densely on $\bar{P}^0$, $f$ must have a trivial linear part. This is, $f$ acts as a translation, i.e. as an element of $\R^q \times \bar{P}^0$,  on  the orbit $\R^q  x\times \bar{P}^0 y$. So the orbit goes down to a torus in $M$. (A priori, $f$ does not necessarily belong to $\R^q \times \bar{P}^0$).
  
This whole discussion proves Theorem~\ref{FoliationToriiThm}.


\begin{thebibliography}{10}

\bibitem{AC}
M.M. Alexandrino, F.C., Jr. Caramello, {\sl Leaf closures of Riemannian foliations: a survey on topological and geometric aspects of Killing foliations}. Expo. Math. {\bf 40} (2), 177--230 (2022).

\bibitem{ALC}
J. A. Álvarez López, A. Candel, {\sl Equicontinuous foliated spaces}. Math. Z. {\bf 263} (4), 725--774  (2009).

\bibitem{ALMG}
J.A. Álvarez López, M.F. Moreira Galicia, {\sl Topological Molino's theory}. Pacific J. Math. {\bf 280} (2), 257--314 (2016).

\bibitem{ABM} A. Andrada, V. del Barco, A. Moroianu, {\sl Locally conformally product structures on solvmanifolds}. Ann. Mat. Pura Appl. {\bf 203}, 2425--2456 (2024).

\bibitem{Bar}
T. Barbot, {\sl Plane affine geometry and Anosov flows}. Ann. Sci. École Norm. Sup. (4) {\bf 34} (6), 871--889 (2001).

\bibitem{BFM} F. Belgun, B. Flamencourt, A. Moroianu, {\sl Weyl structures with special holonomy on compact conformal manifolds}. arXiv:2305.06637 (2023).

\bibitem{BM} F. Belgun, A. Moroianu, {\sl On the irreducibility of locally metric connections.} J. reine angew. Math. {\bf 714}, 123--150 (2016).

\bibitem{Blu79}
R.A. Blumenthal, {\sl Transversely homogeneous foliations}. Ann. Inst. Fourier {\bf 29} (4), vii, 143--158  (1979).


\bibitem{BMT}  C. Boubel, P. Mounoud, C. Tarquini, {\sl Lorentzian foliations on $3$-manifolds}. Ergodic Theory and Dynamical Systems, {\bf 26} (5), 1339--1362 (2006).

\bibitem{Car} Y. Carrière, {\sl Les propriétés topologiques des flots riemanniens retrouvées à l'aide du théorème des variétés
presque plates}. Math. Z., {\bf 186} (3), pp. 393--400 (1984).

\bibitem{FlaLCP} B. Flamencourt, {\sl Locally conformally product structures}. Internat. J. Math. {\bf 35} (5), 2450013 (2024).

\bibitem{FlaLCP2} B. Flamencourt, {\sl The characteristic group of conformally product structures}. Annali di Matematica Pura ed Applicata, {\bf 204}, pp. 189--211 (2025).

\bibitem{Frie}
D. Fried, {\sl Closed similarity manifolds}. Comment. Math. Helv. {\bf 55} (4), 576--582 (1980).

\bibitem{Ghys91}
E. Ghys, {\sl Flots transversalement affines et tissus feuilletés}. Analyse globale et physique mathématique, Mém. Soc. Math. France {\bf 46}, 123–150 (1991).

\bibitem{Hea} A. Haefliger, {\it Homotopy and integrability}. Manifolds–Amsterdam 1970 (Proc. Nuffic Summer School), Lect. Notes in Math., Springer, Berlin, {\bf 197}, 133–163 (1971).

\bibitem{IMT}
T. Inaba, S. Matsumoto, N. Tsuchiya, {\sl Codimension one transversely affine foliations}. World Scientific Publishing Co., Inc., River Edge, NJ, 263--293 (1994).

\bibitem{Kel}
M. Kellum, {\sl Uniformly quasi-isometric foliations}. Ergodic Theory Dynam. Systems {\bf 13} (1), 101--122  (1993).

\bibitem{Kou} M. Kourganoff, {\sl Similarity structures and de Rham decomposition}. Math. Ann. {\bf 373}, 1075--1101 (2019).

\bibitem{LM}  F. Laudenbach, G. Meigniez,  \href{http://www.numdam.org/item/JEP_2016__3__1_0/}{\sl Haefliger structures and symplectic/contact structures}. Journal de l’École polytechnique, {\bf 3}, 1--29 (2016).

\bibitem{Mat}
S. Matsumoto, {\sl Affine flows on 3-manifolds}. Mem. Amer. Math. Soc. {\bf 162} (771), vi+94 pp (2003).

\bibitem{MN15} V. Matveev, Y. Nikolayevsky, {\sl A counterexample to Belgun–Moroianu conjecture}. C. R. Math. Acad. Sci. Paris {\bf 353}, 455--457 (2015).

\bibitem{MN17} V. Matveev, Y. Nikolayevsky, {\sl Locally conformally Berwald manifolds and compact quotients of
reducible manifolds by homotheties.} Ann. Inst. Fourier (Grenoble) {\bf 67} (2), 843--862 (2017).

\bibitem{Mol} P. Molino, {\it Riemannian foliations}. Progress in Mathematics. {\bf 73}  Birkhäuser Boston, Inc., Boston, MA (1988).

\bibitem{MP24} A. Moroianu, M. Pilca, {\sl Adapted metrics on locally conformally product manifolds.} Proc. Amer. Math. Soc. {\bf 152}, 2221--2228 (2024).

\bibitem{Plan}
J.F. Plante, {\sl Anosov flows, transversely affine foliations, and a conjecture of Verjovsky} J. London Math. Soc. (2) {\bf 23} (2), 359--362 (1981).

\bibitem{PR93}
R. Ponge, H. Reckziegel, {\sl Twisted products in pseudo-Riemannian geometry}. Geom. Dedicata {\bf 48} (1), 15--25 (1993).

\bibitem{Sca}
B.A. Scárdua, {\sl Transversely affine and transversely projective holomorphic foliations}. Ann. Sci. École Norm. Sup. (4) {\bf 30} (2), 169--204 (1997).

\bibitem{Simic}
S. Simić, {\sl Codimension one Anosov flows and a conjecture of Verjovsky}. Ergodic Theory Dynam. Systems {\bf 17} (5), 1211--1231 (1997).

\bibitem{Wol} R.A. Wolak, {\sl Some remarks on equicontinuous foliations}. Ann. Univ. Sci. Budapest, Eötvös Sect. Math. {\bf 41}, 13--21 (1999) .

\end{thebibliography}
\end{document}